\documentclass{mystyle}
\usepackage{graphicx}
\usepackage{amsfonts, amssymb, amsmath, mathrsfs, latexsym}
\usepackage[all]{xy}
\usepackage[dvipsnames]{xcolor}
\usepackage{mathtools}
\usepackage{tikz}
\usetikzlibrary{positioning}
\usepackage{xcolor}
\usepackage[normalem]{ulem}
\usepackage{booktabs}

\ifpdf
	\usepackage[unicode,pdftex]{hyperref}
	\input glyphtounicode
\else
	\usepackage[unicode,dvipdfm]{hyperref}
\fi

\newtheorem{thm}{Theorem}[section]
\newtheorem{prop}[thm]{Proposition}
\newtheorem{lem}[thm]{Lemma}
\newtheorem{cor}[thm]{Corollary}

\theoremstyle{definition}

\newtheorem{rmk}[thm]{Remark}

\newtheorem*{mcs}{Modified Assertion of Severi}

\newcommand{\PP}{\ensuremath{\mathbb{P}}}

\newcommand{\h}{\ensuremath{\mathcal}}
\newcommand{\HL}{\ensuremath{\mathcal{H}^\mathcal{L}_}}

\newcommand{\cv}{\color{violet}}
\newcommand{\vni}{\vskip 4pt \noindent}

\title[Existence, reducibility of Hilbert scheme of linearly normal curves]
{Existence and reducibility of the Hilbert \\scheme of smooth and  linearly normal curves in $\mathbb{P}^r$ of relatively high degrees}
\thanks{This work was started when the author was enjoying the hospitality and the stimulating atmosphere of the Max-Planck-Insitut f\"ur Mathematik (Bonn). The author was supported in part by National Research Foundation of South Korea (2019R1I1A1A01058457). He  would like to thank the referees for several valuable comments and suggestions which enhanced the clarity as well as the readability of the paper significantly. He is also grateful to Dawei Chen and Claudio Fontanari for  useful remarks on admissible covers.}
\author[Changho Keem]{Changho Keem}
\address{
Department of Mathematics,
Seoul University\\
Seoul 08826,  
South Korea}
\email{ckeem1@gmail.com or ckeem@snu.ac.kr}
\date{\today}
\begin{document}

\begin{abstract}
Let $\mathcal{H}_{d,g,r}$ be the Hilbert scheme parametrizing smooth irreducible and non-degenerate curves of degree $d$ and genus $g$ in $\mathbb{P}^r.$
We denote by $\mathcal{H}^\mathcal{L}_{d,g,r}$ the union of those components of $\mathcal{H}_{d,g,r}$ whose general element is linearly normal. In this article 
we show that $\mathcal{H}^\mathcal{L}_{d,g,r}$ ($d\ge g+r-3$) is non-empty in a certain optimal range of triples $(d,g,r)$ and is empty outside the range.
This settles the existence (or non-emptiness if one prefers) of the Hilbert scheme $\mathcal{H}^\mathcal{L}_{d,g,r}$ 
of linearly normal curves of degree $d$ and genus $g$ in $\PP^r$ for $g+r-3\le d\le g+r$,  $r\ge 3$.
We also determine all the triples $(d,g,r)$ with $g+r-3\le d\le g+r$ for which $\mathcal{H}^\mathcal{L}_{d,g,r}$ is reducible (or irreducible).
\end{abstract}
\maketitle
\section{ An overview, preliminaries and basic set-up}

Given non-negative integers $d$, $g$ and $r\ge 3$, we  denote by $\mathcal{H}_{d,g,r}$ the Hilbert scheme of smooth curves parametrizing smooth irreducible and non-degenerate curves of degree $d$ and genus $g$ in $\PP^r$.
\vni
In two recent works \cite{JPAA, KK3}, the authors studied the Hilbert scheme of {\bf linearly normal} curves in $\PP^r$, $r\ge 3$.  As a preliminary attempt toward a reasonable settlement of  the {\bf Modified Assertion of Severi} -- which was first addressed in {\cv \cite [p. 489]{CS}} and later suggested by A. Lopez in  the review of \cite{Keem} on MathSciNet; AMS Mathematical Reviews MR1221726(95a:14026) -- the main problem treated in \cite{JPAA, KK3} was focused on the irreducibility of the Hilbert scheme of linearly normal curves as follows.
\begin{mcs}  A nonempty $\mathcal{H}^\mathcal{L}_{d,g,r}$ is irreducible for any triple $(d,g,r)$ in the Brill-Noether range  $$\rho (d,g,r)=g-(r+1)(g-d+r)\ge 0,$$ where $\mathcal{H}^\mathcal{L}_{d,g,r}$ is the union of those components of ${\mathcal{H}}_{d,g,r}$ whose general element is linearly normal.
\end{mcs}
\vni
 We briefly make a note of the results in \cite{JPAA} regarding  this business as follows.
\begin{thm}[Ballico et al.] \label{JPA}
\begin{enumerate}
\item[(1)] $\mathcal{H}^\mathcal{L}_{g+r,g,r}$ is non-empty and irreducible.
\item[(2)] $\mathcal{H}^\mathcal{L}_{g+r-1,g,r}$ is  non-empty and irreducible for $g\ge r+1$ and is empty for $g\le r$.
\item[(3)] Every non-empty $\mathcal{H}^\mathcal{L}_{g+r-2,g,r}$ is irreducible.
\item[(4)] 
Every non-empty $\mathcal{H}^\mathcal{L}_{g+r-3,g,r}$ is irreducible for $g\ge 2r+3$
and is non-empty for $g\ge 3r+3$.
\end{enumerate}
\end{thm}

\vni
By the Riemann-Roch formula, there is no complete linear series of degree $d$ and dimension $r$ in case $d\ge g+r+1$. Therefore in this range we have
$$\h{H}^\h{L}_{d,g,r}=\emptyset$$
and the Modified Assertion of Severi makes sense only if $g-d+r\ge 0$.

\vni
We remark that the results (1) and (2) in Theorem \ref{JPA} are  strong, i.e. the range of the genus $g$ with respect to the dimension $r$ of the ambient projective space $\PP^r$ for which $\mathcal{H}^\mathcal{L}_{d,g,r}$ ($d\ge g+r-1$) is non-empty is sharp. On the other hand, 
even though the irreducibility of $\mathcal{H}^\mathcal{L}_{d,g,r}$ holds beyond the Brill-Noether range for $d=g+r-2$, the optimal range of $g$ (with respect to $r$) for which $\mathcal{H}^\mathcal{L}_{g+r-2,g,r}$ is non-empty has not been explicitly obtained in \cite{JPAA}. Instead, the following scattered results have been shown; cf. \cite[Remark 2.4]{JPAA}.
\begin{enumerate}
\item[(1)] $\mathcal{H}^\mathcal{L}_{g+r-2,g,r}=\emptyset$ for $g\le r+2$, $r\ge 3$
\item[(2)] $\mathcal{H}^\mathcal{L}_{g+r-2,g,r}\neq\emptyset$ inside the Brill-Noether range $\rho (g+r-2,g,r)=g-2(r+1)\ge 0$, 
\item[(3)] $\mathcal{H}^\mathcal{L}_{g+r-2,g,r}\neq\emptyset$ for {\bf some values} of $g\ge r+3$ outside the Brill-Noether range.
\end{enumerate}
However for most values of $g$ outside the Brill-Noether range, the existence of a component of $\mathcal{H}^\mathcal{L}_{g+r-2,g,r}$ has been left undetermined.

\vni In this paper  we determine the full range of the genus $g$ for which $\mathcal{H}^\mathcal{L}_{g+r-2,g,r}\neq\emptyset.$ We also get the widest possible (and optimal) range of $g$ for which  $$\mathcal{H}^\mathcal{L}_{g+r-3,g,r}\neq\emptyset.$$ As a by-product, we come up with the complete list of triples $(g+r-3, g,r)$ for which $\mathcal{H}^\mathcal{L}_{g+r-3,g,r}$ is reducible. 

\vni
The organization of this paper is as follows. After we briefly recall some terminologies and preliminaries  in the remainder of this section, we start the next section by showing the existence of $\mathcal{H}^\mathcal{L}_{g+r-2,g,r}$ for every $g\ge r+3$ outside the Brill-Noether range.

\vni
In the subsequent section,  we demonstrate the non-emptiness of  $\mathcal{H}^\mathcal{L}_{g+r-3,g,r}$ {\bf in the  optimal range} beyond the Brill-Noether range. 
We then proceed to determine the irreducibility (or reducibility) of $\mathcal{H}^\mathcal{L}_{g+r-3,g,r}$ for every $g$ for which $\mathcal{H}^\mathcal{L}_{g+r-3,g,r}\neq\emptyset$.  

\vni
For notation and conventions, we almost always follow those in \cite{ACGH} and \cite{ACGH2}; e.g. $\pi (d,r)$ is the maximal possible arithmetic genus of an irreducible and non-degenerate curve of degree $d$ in $\PP^r$.  
Following classical terminology, a linear series of degree $d$ and dimension $r$ on a smooth curve $C$ is usually denoted by $g^r_d$.
A base-point-free linear series $g^r_d$ ($r\ge 2$) on a smooth curve $C$ is called birationally very ample when the morphism 
$C \rightarrow \mathbb{P}^r$ induced by  the $g^r_d$ is generically one-to-one (or birational) onto its image.  A base-point-free linear series $g^r_d$ on $C$  is said to be compounded of an involution (compounded for short) if the morphism induced by the linear series gives rise to a non-trivial covering map $C\rightarrow C'$ of degree $k\ge 2$. 
Throughout we work over the field of complex numbers.

\vni
We recall several fundamental results which are rather well-known; cf. \cite{ACGH2}  or \cite[\S 1 and \S 2]{AC2}.
Let $\mathcal{M}_g$ be the moduli space of smooth curves of genus $g$. Given an isomorphism class $[C] \in \mathcal{M}_g$ corresponding to a smooth irreducible curve $C$, there exist a neighborhood $U\subset \mathcal{M}_g$ of the class $[C]$ and a smooth connected variety $\mathcal{M}$ which is a finite ramified covering $h:\mathcal{M} \to U$, as well as  varieties $\mathcal{C}$, $\mathcal{W}^r_d$ and $\mathcal{G}^r_d$ proper over $\mathcal{M}$ with the following properties:
\begin{enumerate}
\item[(1)] $\xi:\mathcal{C}\to\mathcal{M}$ is a universal curve, i.e. for every $p\in \mathcal{M}$, $\xi^{-1}(p)$ is a smooth curve of genus $g$ whose isomorphism class is $h(p)$,
\item[(2)] $\mathcal{W}^r_d$ parametrizes the pairs $(p,L)$ where $L$ is a line bundle of degree $d$ and $h^0(L) \ge r+1$ on $\xi^{-1}(p)$,
\item[(3)] $\mathcal{G}^r_d$ parametrizes the couples $(p, \mathcal{D})$, where $\mathcal{D}$ is possibly an incomplete linear series of degree $d$ and dimension $r$ on $\xi^{-1}(p)$.
\end{enumerate}

\vni
Let $\widetilde{\mathcal{G}}$ ($\widetilde{\mathcal{G}}_\mathcal{L}$,  respectively) be  the union of components of $\mathcal{G}^{r}_{d}$ whose general element $(p,\mathcal{D})$ corresponds to a very ample (very ample and complete, respectively) linear series $\mathcal{D}$ on the curve $C=\xi^{-1}(p)$. By recalling that an open subset of $\mathcal{H}_{d,g,r}$ consisting of elements corresponding to smooth irreducible and non-degenerate curves is a $\mathbb{P}\textrm{GL}(r+1)$-bundle over an open subset of $\widetilde{\mathcal{G}}$, the irreducibility of $\widetilde{\mathcal{G}}$ guarantees the irreducibility of $\mathcal{H}_{d,g,r}$. Likewise, the irreducibility of $\widetilde{\mathcal{G}}_\mathcal{L}$ ensures the irreducibility of 
 $\mathcal{H}_{d,g,r}^\mathcal{L}$.

\vni
The following facts regarding the schemes $\mathcal{G}^{r}_{d}$ and $\mathcal{W}^r_d$ are also well known; cf. \cite[2.a]{H1} and \cite[Ch. 21, \S 3, 5, 6, 11, 12]{ACGH2}. 

\begin{prop}\label{facts}
For non-negative integers $d$, $g$ and $r$, let $$\rho(d,g,r):=g-(r+1)(g-d+r)$$ be the Brill-Noether number.
	\begin{enumerate}
	\item[\rm{(1)}] The dimension of any component of $\mathcal{G}^{r}_{d}$ is at least $$\lambda(d,g,r):=3g-3+\rho(d,g,r).$$ Moreover, if $\rho(d,g,r)\geq0$, there exists a unique component $\mathcal{G}_0$ of $\widetilde{\mathcal{G}}$ which dominates $\mathcal{M}$(or $\mathcal{M}_g$).
		\item[\rm{(2)}] $\mathcal{G}^{1}_{d}$ is smooth and irreducible of dimension $\lambda(d,g,1)$ if $g>1, d\ge 2$ and $d\le g+1$.
	\end{enumerate}
\end{prop}

\vni
 Occasionally we will use the following upper bound of the dimension of an irreducible component of $\mathcal{W}^r_d$; cf. \cite{I}. 

\begin{prop}[\rm{\cite[Proposition 2.1]{I}}]\label{wrdbd}
Let $d,g$ and $r\ge 2$ be positive integers such that  $d\le g+r-2$ and let $\mathcal{W}$ be an irreducible component of $\mathcal{W}^{r}_{d}$. For a general element $(p,L)\in \mathcal{W}$, let $b$ be the degree of the base locus of the line bundle $L$ on $C=\xi^{-1}(p)$. Assume further that for a general $(p,L)\in \mathcal{W}$ the curve $C=\xi^{-1}(p)$ is not hyperelliptic. If the moving part of $L$ is
	\begin{itemize}
	\item[\rm{(a)}] very ample and $r\ge3$, then 
	$\dim \mathcal{W}\le 3d+g+1-5r-2b$;
	\item[\rm{(b)}] birationally very ample, then 
	$\dim \mathcal{W}\le 3d+g-1-4r-2b$;
	\item[\rm{(c)}] compounded, then 
	$\dim \mathcal{W}\le 2g-1+d-2r$.
	\end{itemize}
\end{prop}

\begin{rmk}\label{principal}
\begin{enumerate}
\item[(i)]
In the Brill-Noether range $\rho (d,g,r)\ge 0$, the unique component $\mathcal{G}_0$ of $\widetilde{\mathcal{G}}$ (and the corresponding component $\mathcal{H}_0$ of $\mathcal{H}_{d,g,r}$ as well) which dominates $\mathcal{M}$ or $\mathcal{M}_g$ is called the  ``principal component".  

\vskip 4pt
\noindent
\item[(ii)] In the range $d\le g+r$ inside the Brill-Noether range $\rho (d,g,r)\ge 0$, the principal component $\mathcal{G}_0$ which has the expected dimension $\lambda (d,g,r)$ is one of the components of $\widetilde{\mathcal{G}}_\mathcal{L}$ (cf. \cite[2.1 page 70]{H1})  and  hence $\widetilde{\mathcal{G}}_\mathcal{L}$ and 
 $\mathcal{H}_{d,g,r}^\mathcal{L}$ are {\bf always non-empty} in the Brill-Noether range with $g-d+r\ge 0$.
\end{enumerate}
\end{rmk}
\vni
Inside the family of plane curves of degree $d$ in $\mathbb{P}^2$ which is parametrized naturally by the projective space $\mathbb{P}^N$ with $N=\frac{d(d+3)}{2}$, let $\Sigma_{d,g}\subset \PP^N$ be the Severi variety of plane curves of degree $d$ and geometric genus $g$. 
Denoting by $\mathcal{G'}\subset \mathcal{G}^{2}_{d}$ the 
union of components whose general element $(p,\mathcal{D})$  is such that $\mathcal{D}$ is birationally very ample  on $C=\xi^{-1}(p)$, we 
remark that an open subset of the Severi variety $\Sigma_{d,g}$ is a $\mathbb{P}\textrm{GL}(3)$-bundle over an open subset of  $\mathcal{G}'$.  Therefore the irreducibility of  $\Sigma_{d,g}$ implies the 
irreducibility of the locus $\mathcal{G'}\subset \mathcal{G}^{2}_{d}$ and vice versa. 
Indeed the following theorem of Harris is fundamental for our study; cf. \cite[Theorem 10.7 and 10.12]{ACGH2} or \cite[Lemma 1.1, 1.3 and 2.3]{H2}.

\begin{thm}\label{severi}
$\Sigma_{d,g}$ is irreducible of dimension $3d+g-1=\lambda(d,g,2)+\dim\mathbb{P}\textrm{GL}(3)$.  Equivalently, 
$\mathcal{G}'\subset \mathcal{G}^{2}_{d}$ is irreducible of dimension $\lambda (d,g,2).$
\end{thm}

\vni Our main approach will be to look at the residual linear series of a very ample and complete linear series $g^r_d$ and leverage results on the irreducibility of the Hurwitz spaces and Severi varieties.  As a matter of fact, the cases accessible by this strategy are precisely the cases $d\ge g+r-3$, which is why these cases are the focus of the present paper. 

\vni

\vni As was pointed out by the referee, the irreducibility of Hurwitz spaces and Severi varieties is the only essential point at which the characteristic zero
assumption enters the proof.
In this regard, it is worthwhile to let the readers know that there is a preprint by Christ, He and Tyomkin claiming  the irreducibility of Severi varieties in positive characteristic; cf. \cite{positiveseveri}. See also an interesting  related YouTube video at https://youtu.be/4HAOitlDeKw.

\section{Existence of $\mathcal{H}^\mathcal{L}_{g+r-2,g,r}$}

In this section we prove the following theorem which provides the full range of the genus $g$ for which $\mathcal{H}^\mathcal{L}_{g+r-2,g,r}\neq\emptyset$ and irreducible.

\begin{thm}  \label{g+r-2}\begin{enumerate}
\item[(i)]$\mathcal{H}^\mathcal{L}_{g+r-2,g,r}$ is non-empty, irreducible, generically smooth of the expected dimension $$3g-3+\rho (g+r-2,g,r)+\dim\PP\textrm{GL}(r+1)$$  and has the expected number of moduli for every $g\ge r+3$. 
\item[(ii)] $\mathcal{H}^\mathcal{L}_{g+r-2,g,r}=\emptyset$
if $g\le r+2$.
\end{enumerate}
\end{thm}
\begin{proof} (i) The irreducibility of a non-empty $\mathcal{H}^\mathcal{L}_{g+r-2,g,r}$ was  shown already in \cite[Theorem 2.3]{JPAA} and the existence in the Brill-Noether range $\rho (g+r-2,g,r)=g-2(r+1)\ge 0$ is assured by Remark \ref{principal} (ii). Therefore we work in the range 
$g\le 2r+1$ which is outside the Brill-Noether range.

\vni
We briefly recall the procedure in \cite[Theorem 2.3]{JPAA} showing the irreducibility of $\widetilde{\mathcal{G}}_\mathcal{L}$  (and hence the irreducibility of $\h{H}^\h{L}_{g+r-2,g,r}$) under the assumption that there exists of a component  ${\mathcal{G}}\subset \widetilde{\mathcal{G}}_\mathcal{L}\subset \widetilde{\mathcal{G}}\subset\mathcal{G}^{r}_{g+r-2}$.
The dimension estimate  $$\dim\h{G}=\lambda (g+r-2,g,r)=4g-2r-5$$ is carried out by the inequalities in 
Proposition  \ref{wrdbd} (a) and Proposition \ref{facts} (1). 
Let $\mathcal{W}\subset \mathcal{W}^{r}_{g+r-2}$ be the component containing the image of the natural rational map 
${\mathcal{G}}\overset{\iota}{\dashrightarrow} \mathcal{W}^{r}_{g+r-2}$ with $\iota (\mathcal{D})=|\mathcal{D}|$.
Denoting by 
$\mathcal{W}^\vee\subset \mathcal{W}^{1}_{g-r}$ the locus consisting  of the residual  series (with respect to the canonical series on the corresponding curve) 
of those elements in $\mathcal{W}$, i.e. $$\mathcal{W}^\vee =\{(p, \omega_C\otimes L^{-1}): (p, L)\in\mathcal{W}\},$$ we let $\widetilde{\mathcal{W}}^\vee\subset\mathcal{W}^1_{g-r}$ be the union of those  components $\mathcal{W}^\vee$ corresponding to  $\mathcal{W}$ arising from each component ${\mathcal{G}}\subset \widetilde{\mathcal{G}}_\mathcal{L}$.
By the above dimension estimate, we have
\begin{equation*}\label{scriptpencil}\dim \mathcal{W}^{\vee}=\dim\mathcal{W}=\dim\mathcal{G}=4g-2r-5=\lambda (g-r,g,1)=\dim \mathcal{G}^1_{g-r}.\end{equation*}
It was shown eventually that  $\widetilde{\mathcal{W}}^{\vee}$ is birationally equivalent to  the irreducible locus $\mathcal{G}^1_{g-r}$ (cf. Proposition \ref{facts}(2)), hence $\widetilde{\mathcal{W}}^{\vee}$ is irreducible and so is $\widetilde{\mathcal{G}}_\mathcal{L}$. 

\vni
In order to demonstrate the existence of a component ${\mathcal{G}}\subset \widetilde{\mathcal{G}}_\mathcal{L}\subset \widetilde{\mathcal{G}}\subset\mathcal{G}^{r}_{g+r-2}$, it is enough to show that for a general $(p, g^1_{g-r})\in\h{G}^1_{g-r}$ -- which is non-empty (cf. \cite[Proposition 6.8, p.811]{ACGH2}) -- the residual 
series $\h{D}=|K_C-g^1_{g-r}|$ is a very ample (and complete) $g^r_{g+r-2}$ on $C=\xi^{-1}(p)$. 
Note that $\h{D}=|K_C-g^1_{g-r}|$ is very ample if and only if $$\dim|g^1_{g-r}+u+v|=\dim|g^1_{g-r}|=1$$ for every $u+v\in C_2$. Since $$\rho(g+r-2,g,r)=\rho(g-r,g,1)=g-2(r+1)<0,$$  $\h{G}^1_{g-r}$ dominates the proper irreducible closed sub-locus $\h{M}^1_{g,g-r}\subset\h{M}_g$ generically consisting of general $(g-r)$-gonal curves with a unique base-point-free and complete pencil $g^1_{g-r}$.  Recall that the Clifford index $e$ of  a general $(e+2)$-gonal curve of genus $g\ge 2e+2$ is computed only by the unique pencil computing the gonality as long as $e\neq 0$, i.e. there does not exist a $g^s_{2s+e}$ with $2s+e\le g-1$, $s\ge 2$; cf. \cite[Theorem]{B} or 
\cite[Corollary 1]{KKS}. 
Taking $e+2=g-r$, $\deg|g^1_{g-r}+u+v|=g-r+2=e+4$ and hence $\dim|g^1_{g-r}+u+v|=1$ for any $u+v\in C_2$. 

\vni Finally one may argue that $\mathcal{H}^\mathcal{L}_{g+r-2,g,r}$ is generically smooth as follows. Note that a component $\mathcal{H}$  of the Hilbert scheme $\mathcal{H}^\h{L}_{d,g,r}$ is generically smooth  if and only if the the corresponding component ${\mathcal{G}}\subset \widetilde{\mathcal{G}}_\mathcal{L}\subset \widetilde{\mathcal{G}}\subset\mathcal{G}^{r}_{d}$ is generically smooth. Since we have a sequence of birational maps 
$$\mathcal{G} = \widetilde{\mathcal{G}_{\h{L}}}\dashrightarrow\mathcal{W}\dashrightarrow{\mathcal{W}}^\vee\overset{\kappa}{\dashrightarrow}\mathcal{G}^1_{g-r}  ,$$
$\mathcal{G}$ is birationally equivalent to $\mathcal{G}^1_{g-r}$ which is generically smooth by Proposition \ref{facts} (2); cf. \cite[Proosition 6.8, p.811]{ACGH2} or \cite[Proposition 2.7]{AC2}. Therefore
$\mathcal{H}^\mathcal{L}_{g+r-2,g,r}$ is generically smooth. 

\vni
We
recall that a component $\h{H}\subset \h{H}_{d,g,r}$ is said to have the expected number of moduli if the image of natural functorial map $\h{H}\stackrel{\pi}{\dasharrow}\h{M}_g$ has dimension $$\min(3g-3, 3g-3+\rho(d,g,r)).$$

\vni
In our current situation, the image of the irreducible $\mathcal{H}^\mathcal{L}_{g+r-2,g,r}$ under the functorial map $\pi$ is $\h{M}^1_{g,g-r}$ which has dimension 
$$3g-3+\rho(g-r,g,1)=3g-3+\rho(g+r-2,g,r)=\dim\h{G}.$$
\vni
The statement (ii) is the just the repetition of the statement  \cite[Remark 2.4 (i)]{JPAA} which follows easily from Castelnuovo genus bound. 
\end{proof} 

\begin{rmk}
\begin{enumerate}
\item[(i)] For some particular genus $g$ such as $g=2r+1$ or $g=2r$ which is just below the Brill-Noether range, one may show the 
existence of $\mathcal{H}^\mathcal{L}_{g+r-2,g,r}$ utilizing the existence of a certain reducible and smoothable curve $C_0\subset \PP^r$ of degree $d=g+r-2$ and arithmetic genus $g$; cf. \cite[Corollary 1.3, Remark 1.5]{BE}. 

\item[(ii)] For $r=3$ and  $d=g+r-2=g+1$, it is known that $$\mathcal{H}^\mathcal{L}_{g+1,g,3}=\mathcal{H}_{g+1,g,3}\neq\emptyset$$ and is irreducible of the expected dimension $4(g+1)$ for $g\ge 6$, whereas  $\mathcal{H}_{g+1,g,3}=\emptyset$ for $g\le 5$; cf. \cite[Proposition 2.1]{KKL} and references therein.

\item[(iii)] For $r=4$ and $d=g+r-2=g+2$, it is also known that $$\mathcal{H}^\mathcal{L}_{g+2,g,4}=\mathcal{H}_{g+2,g,4}\neq\emptyset$$  and is  irreducible for $g\ge 7$ and is empty for $g\le 6$; cf. \cite[Corollary 2.2]{KKy2}. 

\item[(iv)] For $r=5$, it is not clear at all if $\mathcal{H}^\mathcal{L}_{g+3,g,5}$ (which we know its irreducibility for $g\ge 8$ by Theorem \ref{g+r-2}) is the only component of $\mathcal{H}_{g+3,g,5}$.
As far as the author knows,  the irreducibility of $\mathcal{H}_{d,g,5}$ even for $d\ge g+5$ has not been completely settled  yet.

\item[(v)] It is worthwhile to remark that a component $\h{H}\subset\HL{d,g,r}$ which has the expected dimension does not necessarily have the expected number of moduli or  vice versa; cf.  Theorem \ref{reducible},  Table 2 and Remark \ref{triple}.
\end{enumerate}
\end{rmk}

\section{Existence  of $\h{H}^\h{L}_{g+r-3,g,r}$ and irreducibility}

In this section we turn to the next case $d=g+r-3$ and extend the existence of the Hilbert scheme $\mathcal{H}^\mathcal{L}_{g+r-3,g,r}$ of linearly normal curves beyond the Brill-Noether range. We show that there exists an explicit range $\Upsilon\subset \mathbb{N}^{\oplus 3}$ -- which will be specified later at the end of the section -- such that 
$$\HL{d,g,r}\neq\emptyset \textrm{ if and only if } (d,g,r)=(g+r-3,g,r)\in\Upsilon\subset \mathbb{N}^{\oplus 3}.$$ We also determine the reducibility (or irreducibility)  of  $\mathcal{H}^\mathcal{L}_{g+r-3,g,r}$ for every $(g+r-3,g,r)\in\Upsilon$ especially  in the range $$r+5\le g\le 2r+2$$ whereas every non-empty $\mathcal{H}^\mathcal{L}_{g+r-3,g,r}$ is known to be irreducible for $g\ge 2r+3$; cf. Theorem \ref{JPA} or \cite[Theorem 2.5]{JPAA}. 
In several places in the proofs - without explicit mention - we use the basic fact that hyperelliptic curves do not carry linear series  which are both special and very ample.  
We now collect several  relevant known results regarding the Hilbert scheme of smooth curves in $\PP^3$ and $\PP^4$.

\begin{rmk} \label{r=34}
\begin{enumerate}
\item[(i)] Somewhat stronger results hold for curves in $\PP^3$. For $r=3$ and $d=g+r-3=g$,  by \cite[Theorem 2.1]{KKy1}  
$$\h{H}_{g,g,3}= \h{H}^\h{L}_{g,g,3} ~\text{is irreducible for} ~g\ge 8$$ and is empty for $g\le 7$. Even though the existence of $\h{H}_{g,g,3}$ for $g\ge 8$ was not explicitly  mentioned in \cite{KKy1}, the existence of a smooth irreducible curve in $\PP^3$ of degree $d=g$ for $g\ge 9$ is assured by a result due to Gruson-Peskine \cite{Gruson}; note that $g\le \pi_1(g,3)$ if $g\ge 9$, where $\pi_1(d,r)$ is so-called the second Castelnuovo genus bound for curves of degree $d$ in $\PP^3$; cf. \cite[page 97]{H1} or \cite[page 123]{ACGH}.  For $g=8$, there are smooth curves of type $(3,5)$ on a smooth quadric surface which form an irreducible family of the expected dimension. 

\item[(ii)] For $r=4$ and $d=g+r-3=g+1$, the irreducibility of $\mathcal{H}_{g+1,g,4}$ has not been fully determined  yet to the best knowledge of the author. It is only known that a non-empty $\mathcal{H}_{g+1,g,4}^\mathcal{L}$ is irreducible unless $g=9$ by \cite[Theorem 2.1]{KK3}.  Regarding the existence of $\mathcal{H}_{g+1,g,4}^\mathcal{L}$ and the irreducibility of $\mathcal{H}_{g+1,g,4}$ as well -- especially outside the Brill-Noether range  $\rho (g+1,g,4)=g-15<0$ -- the followings are known; cf. \cite[Theorem 2.2]{KK3}. 

\noindent
\vni
(a) $\mathcal{H}_{g+1,g,4}=\mathcal{H}_{g+1,g,4}^\mathcal{L}=\emptyset$ for $g\leq8$.		

\vni
(b) For $g=9$, $\mathcal{H}_{10,9,4}=\mathcal{H}_{10,9,4}^\mathcal{L}\neq\emptyset$ and is reducible with two components of dimensions $42$ and $43$.

\vni
(c) For $g=10$,  $\mathcal{H}_{11,10,4}=\mathcal{H}_{11,10,4}^\mathcal{L}\neq\emptyset$ and is irreducible of the expected dimension $46$.

\vni
(d) For $g=12$, $\mathcal{H}_{13,12,4}\neq\emptyset$ and is reducible with two components of the same expected dimension $54$, whereas 
		$\mathcal{H}_{13,12,4}^\mathcal{L}$ is irreducible; note that in this case 
		$\mathcal{H}_{13,12,4}^\mathcal{L}\neq\mathcal{H}_{13,12,4}$.
		
\vni
(e) For the other low genus $g=11, 13, 14$ outside the Brill-Noether range,
$$\h{H}_{g+1,g,4}=\h{H}^\h{L}_{g+1,g,4}\neq\emptyset$$ and is irreducible, generically smooth  of the expected dimension by \cite[Theorem 2.7]{JAA}, whose existence is guaranteed by \cite[Theorem 0.1]{rathman}.
\end{enumerate}
\end{rmk}

\vni
Having collected all the relevant results concerning the existence and the irreducibility of $\HL{g+r-3,g,r}$ for $r=3,4$ in the above Remark \ref{r=34}, we now shift our attention to curves in $\PP^r$ with $r\ge 5$, which we always assume in this section unless otherwise specified.
By the Castelnuovo genus bound and \cite[Theorem 2.5 and Remark 2.6]{JPAA}, 
we know the following: 

\vni
\begin{itemize}
\item{}$\h{H}^\h{L}_{g+r-3,g,r}=\h{H}_{g+r-3,g,r}=\emptyset$ for $g\le r+4$.

\item{}Every non-empty $\h{H}^\h{L}_{g+r-3,g,r}$ is irreducible for $g\ge 2r+3$.
\end{itemize}

\vni
For low genus $g=r+5$ which is just above the empty range $g\le r+4$,  $$\h{H}^\h{L}_{g+r-3,g,r}= \h{H}_{g+r-3,g,r}\neq\emptyset$$ consisting of extremal curves (if $r\ge 4$) as one expects. On the other hand,  the irreducibility of the Hilbert scheme of extremal curves $$\h{H}^\h{L}_{g+r-3,g,r}=\h{H}^\h{L}_{2r+2,r+5,r}=\h{H}_{2r+2,r+5,r}$$ is known to be dependent on the dimension $r$ of the ambient projective space $\PP^r$. Note that for a general $\h{D}\in\h{G}\subset\h{G}^r_{2r+2}$, the residual series $\h{E}=\h{D}^\vee=g^2_6$ is compounded  and $\h{E}=2g^1_3$ if $r\ge 6$. For $r=5$, either $\h{E}=2g^1_3$ or $\h{E}$ is very ample.
We  quote this as the following remark.

\vni
\begin{rmk} \label{extremal}
\begin{enumerate}
\item[(i)] For $r\ge 6$, $\h{H}^\h{L}_{2r+2,r+5,r}=\h{H}_{2r+2,r+5,r}\neq\emptyset$ and is irreducible. 

\item[(ii)]
For $r=5$, $\h{H}^\h{L}_{2r+2,r+5,r}=\mathcal{H}_{2r+2,r+5,r}\neq\emptyset$ and is reducible with two components; one consisting of the images of smooth plane sextics under the $2$-tuple embedding onto a Veronese surface and the other one consisting of trigonal curves on a rational normal scroll. 
\end{enumerate}
\end{rmk}
\vni
In the next case $g=r+6$, we are confronted  with the following seemingly unexpected phenomenon.
Indeed what we shall see next is that there exist triples $(g+r-3,g,r)$ for which $$\h{H}_{g+r-3, g,r}=\emptyset$$
beyond the obvious empty range $g\le r+4$.  The following is the first such kind the author knows of. However, there are further  examples of this kind (and even more) 
if the degree and genus of curves becomes relatively lower with respect to the dimension of the ambient projective space.

\begin{prop} \label{empty}
\begin{enumerate}
\item[(i)]$\h{H}_{g+r-3, g,r}=\emptyset$
for $g=r+6$ and $r\ge 10$. 
\item[(ii)] For $g=r+6$ and $5\le r\le 9$, $$\h{H}^\h{L}_{g+r-3, g,r}=\h{H}_{g+r-3, g,r}\neq\emptyset$$ which is irreducible of the expected dimension 
with the expected number of moduli.

\end{enumerate}
\end{prop}
\begin{proof} (i) Let $C\subset\PP^r$ ($r\ge 10$) be a smooth, non-degenerate and irreducible curve of genus $g=r+6$ with  very ample hyperplane series $\h{D}=g^r_{2r+3}$. We may assume that $\h{D}$ is complete, otherwise we have 
$$g=r+6\le\pi (2r+3, r+1)=r+4$$ which is an absurdity. Consider the residual series $\h{E}=|K_C-\h{D}|=g^2_7$. Note that 
$$g=r+6 \gneqq \frac{(e-1)(e-2)}{2}$$  
for any $e\le 7$. Hence $\h{E}$ is compounded, with the non-empty  base locus $\Delta=\{q\}$ and $C$ is either 
\begin{enumerate}
\item[(a)]
 trigonal with $\h{E}=|2g^1_3+q|$ or 
 \item[(b)] bielliptic with $\h{E}=|\phi^*(g^2_3)+q|$, where $C\stackrel{\phi}{\rightarrow} E$ is the double covering of an elliptic curve $E$. 
 \end{enumerate}
 If $C$ is trigonal,  let  $q+t+s\in g^1_3$ be the unique trigonal divisor containing the base point $q$. We then have$$\dim|\h{D}-t-s|=\dim|K-2g^1_3-q-t-s|=\dim| K-3g^1_3|\ge r-1.$$ 
 \vni
 If $C$ is bielliptic,
 $$\dim|\h{E}+\phi^*(u)|=\dim|\phi^*(g^2_3+u)+q|\ge 3,$$ for any $u\in E$, implying $$\dim|\h{D}-\phi^*(u)|=\dim|K-\h{E}-\phi^*(u)|=\dim|K_C-\phi^*(g^2_3+u)-q|\ge r-1, $$ 
which is a contradiction since $\h{D}$ is very ample. 
 
 \vni
(ii)  For $5\le r\le 9$,  $\h{H}^\h{L}_{2r+3, r+6,r}=\h{H}_{2r+3, r+6,r}$ is irreducible of the expected dimension 
$$\lambda (2r+3,r+6,r) +\dim\PP\textrm{GL}(r+1)$$ and the proof is rather similar to the proof of \cite[Theorem 2.5]{JPAA}, which asserts the irreducibility of $\HL{g+r-3,g,r}$ for $g\ge 2r+3$. However, we present a shortened version of a proof for the convenience of readers. Let  $\widetilde{\mathcal{G}}_\mathcal{L}$ be the union of irreducible components $\mathcal{G}$ of $\mathcal{G}^{r}_{g+r-3}$ whose general element corresponds to a pair $(p,\mathcal{D})$ such that $\mathcal{D}$ is very ample and complete  linear series on $C:=\xi^{-1}(p)$ -- as was defined in the paragraph preceding Proposition \ref{facts}. Assuming the existence of a  component ${\mathcal{G}}\subset \widetilde{\mathcal{G}}_\mathcal{L}\subset \widetilde{\mathcal{G}}\subset\mathcal{G}^{r}_{g+r-3}$, we note that the moving part of a  general element  $$\h{E}\in\mathcal{W}^{\vee}\subset\mathcal{W}^{2}_{g-r+1}=\h{W}^2_7$$ (where $\h{W}^\vee$ is defined as in the proof of Theorem \ref{g+r-2}) is birationally very ample by the proof of (i)  and is base-point-free since 
$$g=r+6\ge 11\gneqq \frac{(6-1)(6-2)}{2}.$$  Using Proposition \ref{wrdbd}(b), 
one has the exact dimension estimate 
\begin{equation}\label{equall}
	\dim\mathcal{G}=\dim\mathcal{W}=\dim\mathcal{W}^{\vee}=4g-3r-6=\lambda (g+r-3,g,r).
\end{equation}

\noindent
Let $\widetilde{\mathcal{W}}^\vee$ be the union of the  components $\mathcal{W}^\vee$ of $\mathcal{W}^2_{g-r+1}$ corresponding to each $\mathcal{G}\subset\widetilde{\mathcal{G}}_\mathcal{L}$.
We also let $\mathcal{G}'$ be the union of irreducible components of $\mathcal{G}^{2}_{g-r+1}$ whose general element is a base-point-free and birationally very ample linear series.  We recall that, by Theorem \ref{severi}, $\mathcal{G'}$ is irreducible and $$\dim \mathcal{G'}=3(g-r+1)+g-9=4g-3r-6.$$
By  the equality (\ref{equall}), \begin{equation}\label{dominantt}\dim \mathcal{W}^{\vee}=\dim\mathcal{G}=4g-3r-6=\dim \mathcal{G}'.\end{equation}
Then one may argue that there is a natural generically injective dominant rational map $\widetilde{\mathcal{W}}^\vee\overset{\kappa}{\dashrightarrow}\mathcal{G}'$ with $\kappa(|\mathcal{E}|)=\mathcal{E}$ as well as 
another natural rational map 
$\mathcal{G}' \overset{\iota}{\dashrightarrow}\widetilde{\mathcal{W}}{^\vee}$ with $\iota (\mathcal{E})=|\mathcal{E|}$, which is an inverse to $\kappa$ (wherever it is defined). It then follows that $\widetilde{\mathcal{W}}^{\vee}$ is birationally equivalent to  the irreducible locus $\mathcal{G}'$, hence $\widetilde{\mathcal{W}}^{\vee}$ is irreducible and so is $\widetilde{\mathcal{G}}_\mathcal{L}$. Since $\mathcal{H}_{g+r-3,g,r}^\mathcal{L}$ is a $\mathbb{P}\textrm{GL}(r+1)$-bundle over an open subset of $\widetilde{\mathcal{G}}_\mathcal{L}$,  $\mathcal{H}_{g+r-3,g,r}^\mathcal{L}$ is irreducible of the expected dimension.
\vni
As was suggested by the proof of the irreducibility which we outlined above, we show the existence of 
$\h{H}^\h{L}_{g+r-3, g,r}=\h{H}^\h{L}_{2r+3, r+6,r}$  ($5\le r\le 9$)
 as follows.
\begin{enumerate}
\item[(1)] In particular for $r=9$, $$\h{H}^\h{L}_{2r+3, r+6,r}=\h{H}^\h{L}_{21, 15,9}$$ is the family of smooth plane septics embedded in $\PP^9$ by the $3$-tuple embedding; note that $\h{E}=g^2_7$ is birationally very ample and $g=15=p_a(C_\h{E})$ where $C_{\h{E}}\subset\PP^2$ is the plane curve induced by $\h{E}$. 
\item[(2)] The existence of $\h{H}^\h{L}_{2r+3, r+6,r}$ for the other $5\le r\le 8$ is assured by the existence of appropriate nodal plane septics as follows.

\vni We  fix and use the following standard convention for the rest of the paper.  $S_n \stackrel{\pi}{\rightarrow} \PP^2$ is the rational surface $\PP^2$ blown up at $n$-points in general position. Let $(a; b_1,\cdots , b_n)$ denote the linear system $|\pi^*(aL)-\sum_{i=1}^{n}b_iE_i|$ on $S_n$ -- where  $\{E_i\}_{i=1}^n$ are the  exceptional divisors with linear equivalence classes $\{e_i\}_{i=1}^n$
and $L$ is a line in $\PP^2$. 
 Let $l\in \text {Pic}(S_n)$ be the class of $\pi^*(L)$.

\vni
Here we consider the surface $S:=S_{9-r} \stackrel{\pi}{\rightarrow}\PP^2$ ($5\le r\le 8$).  We note that $(7;2,\cdots,2)$ is very ample on $S$; cf. \cite[V 4.12, Ex.4.8]{Hartshorne} and hence contains an irreducible and non-singular member $C$, which is a curve of genus
$$g=\frac{1}{2}(7-1)(7-2)-\frac{1}{2}\sum b_i(b_i-1)=15-(9-r)=r+6$$
and degree
$$d=(3;1,\cdots,1)\cdot(7;2,\cdots,2)=21-2(9-r)=g+r-3$$
in $\PP^r$ under the embedding induced by the anti-canonical linear system $(3;1,\cdots, 1)$. 
\end{enumerate}
\vni The functorial map $\h{H}^\h{L}_{2r+3, r+6,r}\dashrightarrow\h{M}_g$ factors through $\h{G}=\widetilde{\h{G}}_\h{L}\subset\h{G}^r_{2r+3}$ and the functorial map  $\Sigma_{7,r+6}\dasharrow\h{M}_g$ factors through $\h{G}'\subset\h{G}^2_7$. Since
$$\h{G}\stackrel{bir}{\cong}\h{W}\stackrel{bir}{\cong}\h{W}^\vee\stackrel{bir}{\cong}\h{G}',$$ 
the two functorial maps have the same image in $\h{M}_g$. By \cite[Theorem 4.2]{Sernesi}, $\Sigma_{7,r+6}$ has the expected number of moduli and since 
$$\rho (2r+3,r+6,r)=\rho (7,r+6,2),$$
$\h{H}^\h{L}_{2r+3, r+6,r}$ also has the expected number of moduli.
\end{proof}

\vni
Almost all the components of $\h{H}^\h{L}_{d,g,r}$ which we have seen so far in this article have the minimal possible dimension $$\h{X}(d,g,r):=\lambda (d,g,r)+\dim\PP\textrm{GL}(r+1)$$
with the expected number of moduli.

\vni
For the next case $g=r+7$, the following theorem asserts that
$\h{H}^\h{L}_{g+r-3,g,r}=\h{H}^\h{L}_{2r+3, r+7,r}$ is non-empty but reducible for every $r\ge 5$. Moreover if  $r\ge 15$,  no component of $\h{H}^\h{L}_{2r+3, r+6,r}$ has the expected number of moduli,  unlike the previous case $g=r+6$.

\begin{thm}\label{g=r+7} 
\begin{enumerate}
\item[(i)]
For $g=r+7$ and $r\ge 15$,
$$\h{H}^\h{L}_{g+r-3,g,r}=\h{H}^\h{L}_{2r+4,r+7,r}\neq\emptyset$$ and is reducible without any component of the expected dimension. Moreover all the  components 
have more than the expected number of moduli.
\item[(ii)] For $g=r+7$ and $5\le r\le 14$,
$$\h{H}^\h{L}_{g+r-3,g,r}=\h{H}^\h{L}_{2r+4,r+7,r}\neq\emptyset$$ and is reducible  with at least one component of the expected dimension having the expected number of moduli, together with another component having more than the expected number of moduli.
\end{enumerate}
\end{thm}

\begin{proof} (i) We deal with the case $r\ge 15$ first. Let 
$$\h{D}=g^r_{2r+4}\in \h{G}\subset\widetilde{\h{G}}_\h{L}\subset \h{G}^r_{2r+4}$$ be a complete and very ample linear series on a smooth curve $C$ of genus $g=r+7$. 
The residual series $\h{E}=|K_C-\h{D}|=g^2_8$ is compounded since $$g=r+7\gneqq21=\frac{(8-1)(8-2)}{2}\quad \textrm{for} \quad r\ge 15.$$

\vni
Let $\delta$ be the degree of the base locus $\Delta$ of the compounded $\h{E}$. We list up all the possible cases as follows.

\begin{enumerate}
\item[(1)] $\delta =0$, $\h{E}=2g^1_4$ and $C$ is $4$-gonal.
Recall that on a general $4$-gonal curve $C$, there is  a unique $g^1_4$ and 
$$|K_C-2g^1_4|=g^r_{2r+4}$$ is very ample by \cite[Corollary 3.3]{Keem}.  Hence $\h{G}$ dominates the irreducible locus $\h{M}^1_{g,4}$ and $\dim\h{G}=\dim\h{M}^1_{g,4}$.
Since
\begin{equation*}
\lambda (d,g,r)=3g-3+\rho(d,g,r)=r+22\lneqq\dim\h{M}^1_{g,4}=2g+3=2r+17
\end{equation*}
the component $\h{G}$ dominating $\h{M}^1_{g,4}$ as well as the corresponding component of $\h{H}^\h{L}_{2r+4,r+7,r}$ over $\h{G}$ -- which we denote\footnote{Notations of this kind shall be defined precisely in a more general context  in the proof of Theorem \ref{reducible}. } by $\h{H}_{4,0}$-- have more than the expected number of moduli, with the dimension strictly greater than the expected dimension $\h{X}(d,g,r)$.
\item[(2)]  $\delta =0$, $C$ is a double cover $C\stackrel{\pi}{\rightarrow} E$ of a non-hyperelliptic curve          $E$ of genus $h=3$ and $\h{E}=\pi^*(g^2_4)=\pi^*(|K_E|)$.
\vni Let  $\mathcal{X}_{n,h}$ be the locus in $\mathcal{M}_g$ corresponding to curves
which are n-fold coverings of smooth curves of genus $h$. In the current case, 
there exist a sequence of natural rational maps $$\h{G}\dashrightarrow\mathcal{W}\dashrightarrow \mathcal{W}^\vee\dashrightarrow \mathcal{X}_{2,3},$$ which are generically injective into a component of $\h{X}_{2,3}$.

\vni
Recall that for a covering $C\stackrel{\pi}{\rightarrow} E$ and a line bundle $M$ on the base curve $E$ of genus $h$, it is well known by projection formula  that
\begin{eqnarray*} H^0(C, \pi^*M)&=&
H^0(E,\pi_*\pi^*M)=H^0(E, \pi_*(\pi^*M\otimes_{\h{O}_C}\h{O}_C))\\&=&H^0(E,M\otimes_{\h{O}_E}\pi_*\h{O}_C),
\end{eqnarray*}
and that $$\det\pi_*\h{O}_C\cong\h{O}_{E}(-D)$$ for some divisor $D$ on $E$ such that $2D$ is linearly equivalent to the branch divisor $B$ of $\pi$ consisting of  those points on $E$ over which $\pi$ ramifies; cf. \cite[IV, Ex. 2.6]{Hartshorne}). In particular, the vector bundle $\pi_*\h{O}_C$ on $E$ of rank equal to $\deg\pi$ has degree $$-\deg D=-\frac{1}{2}\deg B=(\deg\pi)\cdot(h-1)-(g-1)\le 0.$$
If  $\pi$ is a double covering, the rank two vector bundle $\pi_*\h{O}_C$ splits into the line bundles $\h{O}_{E}$ and $\det\pi_*\h{O}_C$ of degree $2(h-1)-(g-1)=2h-g-1$; cf. \cite{Mumford}. 
In our case,  we have 
$$H^0(C,\pi^*M)=H^0(E, M)\oplus H^0(E, M\otimes\h{O}_{E}(-D)), $$
and if $\deg M<\deg D=g+1-2h$, then $$H^0(E, M\otimes\h{O}_{E}(-D))=0$$ and hence 
\begin{equation}\label{pushpull}
H^0(C,\pi^*M)=H^0(E,M).
\end{equation}

\vni
We first put $M=\h{O}(K_E)$. Since  $\deg M=2\cdot 3-2\le g-2\cdot 3$ by our genus assumption $g=r+7\ge 22$, we have 
$$h^0(C,\pi^*\h{O}(K_{E}))=h^0(E, \h{O}(K_{E}))=h=3.$$
We now show that the complete linear series $|K_C-\pi^*K_{E}|$ is very ample. Take any $T+S\in C_2$ and put
$$s=\pi (S), t=\pi(T), S+S'=\pi^*(s), T+T'=\pi^*(t).$$ We have
$$|\pi^*(K_{E})+S+T|=|\pi^*(K_{E}+s+t)-S'-T'|.$$ Substituting $M=\h{O}(K_{E}+s+t)$ in (\ref{pushpull}), 
$$h^0(C, \h{O}(\pi^*(K_{E}+s+t)))=h^0(E, \h{O}(K_{E}+s+t))=h+1=4.$$
Note that  $|K_{E}+s+t|$ is base-point-free and hence  $|\pi^*(K_{E}+s+t)|$ is also base-point-free. We finally have
\begin{eqnarray*}
\hskip 25pt h^0(C, \h{O}(\pi^*(K_{E})+S+T))=h^0(C, \h{O}(\pi^*(K_{E}+s+t)-S'-T'))\\
\hskip 40pt \le h^0(C,\h{O}(\pi^*(K_{E}+s+t)))-1=h=h^0(C,\h{O}(\pi^*K_{E}))
\end{eqnarray*} and therefore  $|K_C-\pi^*K_{E} |$ is very ample.

\vni
So far we have shown that $\h{G}$ dominates a component of $\h{X}_{2,3}$. 

\vni
By  \cite[Theorem 8.23, p.828]{ACGH2},  $\h{X}_{n, h}$ has pure dimension $$\dim\mathcal{X}_{n,h}=2g+(2n-3)(1-h)-2,$$ and by noting that there is a generically one-to-one natural rational map between $\h{G}$ and a component of $\mathcal{X}_{2,3}$, we have 
$$\dim\h{G}=\dim\h{W}=\dim\mathcal{X}_{2,3}=2g-4.$$
Since 
\begin{equation}\label{s}
r+22=\lambda(g+r-3, g,r)\lneqq\dim\h{G}=\dim\mathcal{X}_{2,3}=2g-4=2r+10,
\end{equation}
there is a component $\h{H}_{(2,3)}\subset \h{H}^\h{L}_{g+r-3,g,r}$ arising from and dominating a component of $\h{X}_{2,3}\subset\h{M}_g$, which has more than the expected 
number of moduli by the strict inequality (\ref{s}). We observe that a general element of $\h{H}_{(2,3)}$ is $6$-gonal whereas a general element of $\h{H}_{4,0}$ is $4$-gonal and $$\dim\h{H}_{4,0}=2r+17+\dim\PP\textrm{GL}(r+1)\gneqq\dim\h{H}_{(2,3)}=2r+10+\dim\PP\textrm{GL}(r+1).$$ 
 By semi-continuity it follows that the locus $\h{H}_{(2,3)}$  is not contained in the boundary of $\h{H}_{4,0}$. This shows that $\h{H}_{4,0}$ and $\h{H}_{(2,3)}$ are distinct irreducible components. 
\end{enumerate}
\vni Now we are only left with the following four possibilities.

\begin{enumerate}
\item[(3)]
$\delta =0$ and $\h{E}$ induces  a degree two morphism $C\stackrel{\eta}{\rightarrow} E\subset\PP^2$ onto a singular quartic curve $E$ of geometric genus $h=1$. In this case,  divisors of degree four cut out by lines in  
$\PP^2$ move in a complete linear series of dimension three on the normalization of $E$ by Riemann-Roch hence  $\dim\h{E}=\dim|\eta^*\h{O}_{\PP^2}(1)|=3$, a contradiction.

\item[(4)]
$\delta =0$ and $\h{E}$ induces  a double covering onto a curve of genus $h=2$.
\item[(5)] $\delta =2$, $\h{E}=|2g^1_3+q+s|$ and $C$ is trigonal. 

\item[(6)] $\delta =2$ and $\h{E}=|\phi^*(g^2_3)+q+s|$ for a double covering $C\stackrel{\phi}{\rightarrow} E$ onto an elliptic curve $E$.
\end{enumerate}
\vni
However for the above three cases (4), (5) and (6), the residual series of the compounded $\h{E}$ is not very ample and we omit the easy verification, which is similar to the proof of Proposition \ref{empty} (i).

\vni 
(ii)
For $5\le r\le 14$, there is a  possibility that $\h{E}=g^2_8\in\h{W}^\vee\subset\h{W}^2_8$ -- which is the residual series of a complete very ample $\h{D}\in\h{G}\subset\h{G}^r_{g+r-3}$ -- is birationally very ample or very ample. In this case we may proceed in a similar way as we did in the proof of Proposition \ref{empty}(ii), to come up
with an irreducible family\footnote{The author apologizes for the complicated notation involved. For later use we make the following definition. Let $\h{H}_{\Sigma_{e,g}}$ be the irreducible family in the  Hilbert scheme of smooth curves $\h{H}^\h{L}_{g+r-3,g,r}$ consisting of curves such that the residual series of the very ample and complete hyperplane series $g^r_{g+r-3}$ is a net cut out by lines on the plane curve corresponding to a general element of the Severi variety $\Sigma_{e,g}$, where $e=g-r+1$.} $\h{H}_{\Sigma_{8,g}}\subset\h{H}^\h{L}_{g+r-3,g,r}$ of minimal possible dimension $\h{X}(d,g,r)$.  If $\h{H}_{\Sigma_{8,g}}\neq\emptyset$, the residual series
of the hyperplane series corresponding to a general element of $\h{H}_{\Sigma_{8,g}}$ is a birationally very ample,  base-point-free net $g^2_8$ inducing a plane curve of degree $8$ with only nodal singularities;  by abuse of terminology, one might describe this situation as {\bf ``the family  $\h{H}_{\Sigma_{8,g}}$ corresponds to the Severi variety $\Sigma_{8,g}$".}

\vni
We now show that the family $\h{H}_{\Sigma_{8,g}}$ is non-empty. The issue is whether the residual series of the
linear system cut out by lines on the plane nodal curve $E$ corresponding to a general member of the Severi variety $\Sigma_{8,r+7}$ is a complete and very ample $g^r_{g+r-3}$. For $r=14$ -- which is a trivial case --  $\Sigma_{8,r+7}=\PP^{\frac{8(8+3)}{2}}$  is the family of smooth plane octics and the residual series of the linear series $\h{E}=|L|$ cut out by lines on a smooth plane octic is the linear series $4\h{E}=|4L|$, which is very ample. Hence $\h{H}_{\Sigma_{8,g}}$ is  the family of smooth plane octics embedded by the $4$-tuple embedding into $\PP^{14}=\PP^{\frac{4(4+3)}{2}}$  as curves of genus $g$ and degree $d$ where $$g=r+7=21=\frac{(8-1)(8-2)}{2}, \quad d=g+r-3=32=4\cdot 8. $$

 \vni
 For the existence of smooth curves in lower dimensional projective space $\PP^r$, $5\le r\le 13$, we first fix two integers $e=8=g-r+1$ and 
 \begin{equation}\label{delta}\delta=\frac{(e-1)(e-2)}{2}-g=21-(r+7)=14-r.
 \end{equation}
 Let $$\Delta =\{p_1, \cdots , p_{\delta}\}\in{\text Sym}^\delta(\PP^2)$$ be a general set of $\delta$ points in $\PP^2$. 
 We note that 
 \begin{equation}\label{AC}\delta=14-r \le {\rm min} \{e(e+3)/6,(e-1)(e-2)/2\} \textrm{~and~}  (e,\delta)\neq (6,9).
 \end{equation}  By a  result due to Arbarello and Cornalba \cite[Theorem 3.2]{AC1} regarding the configuration of nodes on plane nodal curves in  $\Sigma_{e, g}$, there exists a curve $E\in\Sigma_{e,g}$ having nodes at $\Delta$ and no further singularities. With the same choice of a general $\Delta$, there also exists a curve $E'\in\Sigma_{e-1, h}$ of genus $h$ and degree $e-1$ having nodes at $\Delta$ where $h=\frac{(e-2)(e-3)}{2}-\delta=r+1$.

\vni
Let $S:=S_\delta\stackrel{\pi}{\rightarrow}\PP^2$ be the blowing up of $\PP^2$ at  $\Delta$. 
We consider the linear system 
$$|H|=|\pi^*((e-4)L)-\sum_{i=1}^\delta E_i| $$ on $S$.
Recall that by a result of  d'Almeida and Hirschowitz \cite[Theorem 0]{Coppens}, the linear system $|\pi^*(tL)-\sum_{i=1}^\delta E_i|$ on $S$ is very ample if 
\begin{equation}\label{almeida}\delta\le \frac{(t+3)t}{2}-5.
\end{equation} Taking $t=(e-4)$, one easily verifies that the inequality (\ref{almeida}) is equivalent to $r\ge 5$, whence $|H|$ is very ample.
The linear system $$|C|=|\pi^*(eL)-\sum_{i=1}^\delta 2E_i|$$
contains the non-singular model $C$
of the nodal plane curve $E$ and $$H\cdot C=((e-4)l-\sum_{i=1}^\delta e_i)\cdot (el-\sum_{i=1}^\delta 2e_i)=e(e-4)-2\delta=g+r-3.$$
Recall that $\Delta =\{p_1,\cdots ,p_\delta\}\in{\text Sym}^\delta(\PP^2)$ is general and there is a nodal curve $E'\in \Sigma_{e-1, h} $ with nodes at $\Delta$. Therefore we may assume further that $\Delta$ imposes independent conditions on the linear system of curves of any degree $m\ge e-4$; cf. \cite[Exercise 11, p.54]{ACGH} or  \cite[(1.51, p.31)]{H3}.
\vni
Hence by the projection formula, we obtain 
\begin{align*}
h^0(S,\h{O}_S(H))&=h^0(\PP^2, \pi_*(\h{O}_S(H)))=h^0(\PP^2,\pi_*(\pi^*((e-4)L)-\sum_{i=1}^\delta E_i))\\
&=h^0(\PP^2,\h{O}_{\PP^2}((e-4)L-\Delta))=
1+\frac{1}{2}(e-4)(e-1)-\delta\\&=r+1.
\end{align*}
In all, we deduce that the non-singular model $C\subset S$ of a nodal plane curve $E$ of degree $e=g-r+1=8$ in the linear system $|C|=|\pi^*(eL)-\sum_{i=1}^\delta 2E_i|$ is embedded into $\PP^r$ as a curve of degree $d=g+r-3$ by the very ample linear 
system  $|H|=|\pi^*((e-4)L)-\sum_{i=1}^\delta E_i)|.$

\vni
Furthermore,  by our choice of general $\delta$ points $\Delta\subset\PP^2$, which imposes independent conditions on curves of any degree $m\ge e-4$, the linear system $g^2_e=|\pi^*(L)_{|C}|$ cut out on $C$ by lines in $\PP^2$ is complete \cite[Exercise 24, p.57]{ACGH}. Hence it follows that the residual series of $|\pi^*(L)_{|C}|$, which is 
\begin{align*}
|(K_S+C&-\pi^*(L))_{|C}|\\&=|(-\pi^*(3L)+\sum E_i+\pi^*(eL)-2\sum E_i-\pi^*(L))_{|C}|\\
&=|(\pi^*((e-4)L)-\sum E_i)_{|C}|=|H_{|C}|
\end{align*}
is  a complete very ample $g^r_{g+r-3}$ on $C$. Therefore we may deduce that  the family $\h{H}_{\Sigma_{8,g}}$ contains linearly normal smooth curves of degree $d=g+r-3$ and genus $g=r+7$ in $\PP^r$.

\vni
We remark  that the family $\h{H}_{\Sigma_{8,g}}$ is not in the closure of $\h{H}_{4,0}$ or $\h{H}_{(2,3)}$ since a compounded linear series cannot be specialized to a birationally very ample linear series.
Therefore the irreducible family $\h{H}_{\Sigma_{8,g}}$ constitutes a component of $\h{H}^\h{L}_{g+r-3,g,r}$.

\vni By a simple dimension count, the locus $\h{H}_{4,0}$ sitting over the family consisting of the linear series which is residual to $2g^1_4$ (and dominating $\h{M}^1_{g,4}$)
remains a component for every $r\ge 5$ by looking at the inequality;
\begin{equation*}\label{4gonalv}
\lambda (d,g,r)=4g-3r-6=r+22\le\dim\h{M}^1_{g,4}=2g+3=2r+17. 
\end{equation*}  

\vni
We further note that the family of curves $\h{H}_{(2,3)}$ sitting over a component of $\h{X}_{2,3}$ has dimension 
 $$\dim\h{X}_{2,3}+\dim\PP\textrm{GL}(r+1)\lneqq\h{X}(d,g,r)$$ for $5\le r\le 11$ since
$$\dim\mathcal{X}_{2,3}=2g-4=2r+10\lneqq  r+22=\lambda(g+r-3, g,r),$$
 and hence the family $\h{H}_{(2,3)}$ is in the boundary of the component $\h{H}_{\Sigma_{8,g}}$. Hence, for $5\le r\le 11$, we have only two components $\h{H}_{4,0}$, $\h{H}_{\Sigma_{8,g}}\subset\h{H}^\h{L}_{g+r-3,g,r}$. For $12\le r\le 14$, $\h{H}_{(2,3)}$ becomes a separate component different from $\h{H}_{4,0}$ or  $\h{H}_{\Sigma_{8,g}}$ and  there are at least three distinct components $\h{H}_{4,0}$, $\h{H}_{(2,3)}$ and $\h{H}_{\Sigma_{8,g}}$. 
The component $\h{H}_{\Sigma_{8,g}}$ corresponding to the Severi variety (which exists only for $r\le 14$) has the expected number of moduli since $\Sigma_{8,g}$ has the expected number of moduli by \cite[Theorem 4.2]{Sernesi}.
\end{proof}

\begin{rmk} As was mentioned in Remark \ref{r=34} (i) and (ii), the irreducibility of $\h{H}^\h{L}_{g,g,3}=\h{H}_{g,g,3}$ holds for every $g\ge 8$, in particular for $g=r+7=10$.   For $r=4$ and $g=r+7=11$, $\h{H}^\h{L}_{g+1,g,4}=\h{H}_{g+1,g,4}$ is also irreducible. In fact, one checks easily that $$\h{H}_{10,10,3}=\h{H}_{\Sigma_{8.10}} \textrm{~and~} \h{H}_{12,11,4}=\h{H}_{\Sigma_{8.11}}.$$
The argument which we used in the proof of Theorem \ref{g=r+7} (ii) also works for $r=3,4$. Hence both $\h{H}_{g,g,3}$ and $\h{H}_{g+1,g,4}$ have the expected number of moduli for $g=r+7$ and $r=3,4$. We summarize our discussion together with the results  we obtained in Theorem \ref{g=r+7} in the following table. 
\end{rmk} 
\begin{table}[ht]
\caption{Components of $\h{H}^\h{L}_{g+r-3,g,r}$ for $g=r+7$} 
\centering 
\begin{tabular}{c c c c c c} 
\hline\hline 
{\footnotesize Dimension}& {\footnotesize\# of }  & {\footnotesize\# of compts} & {\footnotesize \# of compts} & {\footnotesize \# of compts} & {\footnotesize \# of compts}\\ [0.5ex] 
{\footnotesize of $\PP^r$}  & {\footnotesize compts} & {\footnotesize of expected} & {\footnotesize with expected} & {\footnotesize $\lneqq$ expected} &{\footnotesize $\gneqq$ expected}\\ [0.5ex] 
{}  & {} & {\footnotesize dimension} & {\footnotesize \# of moduli} & {\footnotesize \# of moduli} &  {\footnotesize \# of moduli}\\ [0.5ex] 
\hline 
{\footnotesize $r=3$} & {\footnotesize 1} & {\footnotesize 1} & {\footnotesize 1} & {\footnotesize 0} & {\footnotesize 0}\\
{\footnotesize $r=4$} & {\footnotesize 1} & {\footnotesize 1} & {\footnotesize 1} & {\footnotesize 0}  &{\footnotesize 0}\\ 
{\footnotesize $r=5$} & {\footnotesize 2} & {\footnotesize 2} & {\footnotesize 2} & {\footnotesize 0}  &{\footnotesize 0}\\ 
{\footnotesize $6\le r \le 11$} & {\footnotesize 2}& {\footnotesize 1} & {\footnotesize 1}& {\footnotesize 0} & {\footnotesize 1} \\
{\footnotesize $r=12$} & {\footnotesize $\ge 3$}& {\footnotesize $\ge 2$} & {\footnotesize $\ge 2$} & {\footnotesize $0$} & {\footnotesize 1} \\ 
{\footnotesize $13\le r\le 14$} & {\footnotesize $\ge 3$}  & {\footnotesize $1$}  & {\footnotesize $1$} & {\footnotesize $ 0$}  & {\footnotesize $\ge 2$}  \\
{\footnotesize $r\ge 15$}& {\footnotesize $\ge 2$}  & {\footnotesize $0$} & {\footnotesize $0$} & {\footnotesize $0$} &{\footnotesize $\ge 2$} \\ [1ex] 
\hline 
\end{tabular}
\label{table:nonlin} 
\end{table}

\vni
Recall once again that the existence and the irreducibility of $\h{H}^\h{L}_{g+r-3,g,r}$ is completely known for $r=3, 4$ as we have seen in Remark \ref{r=34}. Therefore we continue to
assume $r\ge 5$. 

\vni\vni
The following Theorem \ref{first} asserts that $\HL{g+r-3,g,r}\neq\emptyset$ for all $g\ge r+5$ and every $r\ge 5$ except for the case $g=r+6$, $r\ge 10$;  recall that  there is no smooth curve with the prescribed genus $g=r+6$ and degree $d=g+r-3$ in $\PP^r$ for any $r\ge 10$ as we showed in Proposition \ref{empty} (i). 
The proof Theorem \ref{first} is based on the following lemma concerning linear series on general $k$-gonal curves. 
\begin{lem}\cite[Proposition 1.1]{CKM}\label{kveryample} Assume $2k-g-2<0$. Let $C$ be a general $k$-gonal curve of genus $g$, $k\ge 2$, $0\le m$, $n\in\mathbb{Z}$ such that 
\begin{equation}\label{veryamplek}
g\ge 2m+n(k-1)
\end{equation}
 and let $D\in C_m$. Assume that there is no $E\in g^1_k$ with $E\le D.$ Then $\dim|ng^1_k+D|=n$.
\end{lem}
\begin{thm}\label{first} Given $r\ge 5$, 
we set
\begin{eqnarray*}\Upsilon =&\{(g+r-3,g,r)| ~g\ge r+5, g\neq r+6, r\ge 5\}\\
&\cup \{(2r+3,r+6,r)| ~ 5\le r\le 9\}\subset\mathbb{N}^{\oplus 3}.
\end{eqnarray*}
Then $$\HL{g+r-3,g,r}\neq\emptyset~ \textrm{ ~~if and only if~ }~ (g+r-3,g,r)\in\Upsilon .$$In other words, 
given $r\ge 5$ and for any $g\ge r+5$ there exists a smooth, non-degenerate and linearly normal curve of genus $g$ 
and of degree $d=g+r-3$ in $\PP^r$ unless $g=r+6$ and $r\ge 10$. 
\end{thm}
\begin{proof} Since the existence (non-existence, respectively) was shown for $g=r+6$ and $5\le r\le 9$ ($r\ge 10$, respectively) we may assume $g\neq r+6$. We only need to exhibit a very ample and complete linear series $\h{D}=g^r_{g+r-3}$ on a certain smooth curve of genus $g$. Since $g\neq r+6$, we have either 

(1) $g-r+1=2k$ for some $k\ge 3$ \qquad or 

(2) $g-r+1=2k+1$ for some $k\ge 4$.
\begin{enumerate}
\item[(1)] Suppose $g-r+1=2k$ for some $k\ge 3$: Note that $$\rho (k,g,1)=2k-g-2=-r-1\lneqq 0$$ and we consider a general $k$-gonal curve $C$. Taking  $n=2, m=2$  in Lemma \ref{kveryample}, the numerical condition (\ref{veryamplek}) is satisfied. Furthermore, there is no $E\in g^1_k$ with $E\le D$ since $k\ge 3$. Hence we have $\dim|2g^1_k+D|=2$ for any $D\in C_2$ by Lemma \ref{kveryample} and therefore $$|K_C-2g^1_k|=g^r_{g+r-3}$$ is very ample.
\item[(2)] Suppose $g-r+1=2k+1$ for some $k\ge 4$: Since  $$\rho (k,g,1)=2k-g-2=-r-2\lneqq 0,$$ we again may consider a general $k$-gonal curve $C$. For a general 
$q\in C$ and for any $r+s\in C_2$, we take $D=q+r+s$, $m=3$ and $n=2$ in Lemma \ref{kveryample}. Again the numerical condition  (\ref{veryamplek}) holds and there is no $E\in g^1_k$ with $E\le D$ just because $k\ge 4$. Hence $\dim |2g^1_k+D|=2$ which implies that $$|K_C-2g^1_k-q|=g^{r}_{2g-2-2k-1}=g^r_{g+r-3}$$ is very ample. 
\end{enumerate}\end{proof} 

\begin{rmk}\label{kdelta}
\begin{enumerate}
\item[(i)] If $k=3$ and $g-r+1=2k$ in the proof (1) of Theorem \ref{first},  i.e. $g=r+5$, we remark that curves detected by our proof are trigonal (and extremal) curves 
lying on a rational normal scroll $S\subset\PP^r$. Such trigonal curves are in the linear system $|3H-(r-5)L|$ where $H$ is a hyperplane section and $L$ is one of the rulings of $S$ whereas the trigonal pencil is cut out  by the rulings of the scroll $S$. Through elementary dimension count, one may verify easily that this family of trigonal curves is indeed dense in the corresponding component of the Hilbert scheme.
\item[(ii)] Let 
\begin{eqnarray*}
\h{F}&:=&\{|K_C-2g^1_k|; C\in\h{M}^1_{g,k}\} ~\textrm{ if }~ g-r+1=2k \textrm{ \quad~~or}\\
&:=&\{|K_C-2g^1_k-q|; C\in\h{M}^1_{g,k}\} ~\textrm{ if }~ g-r+1=2k+1
\end{eqnarray*}
be the family of very ample linear series in $\widetilde{\h{G}}_\h{L}\subset\h{G}^r_{g+r-3}$ over $\h{M}^1_{g,k}$ which we constructed in the proof of Theorem \ref{first}. For $g\le 2r+2$ and $g\neq r+6$, this family $\h{F}$ has dimension equal to
\begin{eqnarray*}
\dim\h{M}^1_{g,k}&=&2g+2k-5=3g-4-r ~\textrm{ if }~ g-r+1=2k\quad\text{or} \\
\dim\h{M}^1_{g,k}+1&=&2g+2k-4=3g-4-r ~\textrm{ if }~ g-r+1=2k+1. 
\end{eqnarray*}
In both cases, we have 
$$\lambda(g+r-3,g,r)=4g-3r-6\le \dim\h{F}=3g-4-r$$ by $g\le 2r+2$. 
\item[(iii)] We also remark that  $$\dim\h{F}=\dim\h{F}^\vee=3g-4-r$$ is same as the maximal possible dimension of an irreducible component of 
$\h{W}^2_{g-r+1}$ whose general element is compounded; cf. Proposition \ref{wrdbd} (c).  Here we use the notation $\h{F}^\vee$ for the family $\{\h{E}^\vee| \h{E}\in \h{F}\}$, where $\h{E}^\vee$ is the residual series of $\h{E}$.
\item[(iv)]
Hence this family $\h{F}$
is not contained in any other irreducible family $\h{F}'\subset\widetilde{\h{G}}_\h{L}\subset\h{G}^r_{g+r-3}$ of strictly bigger dimension such that the residual series of a general element of $\h{F}'$ is compounded. We will see in the next theorem that the family $\h{F}$ as well as some other families of similar kind give rise to several distinct components
of $\h{H}^\h{L}_{g+r-3,g,r}$.
\end{enumerate}
\end{rmk}

\vni
Motivated by Theorem \ref{first}, we show the reducibility of $\mathcal{H}^\mathcal{L}_{g+r-3,g,r}$ for every genus $g$ in the range $r+7\le g\le 2r+2$.  This generalizes the reducible example in \cite[Example 2.7]{JPAA} and shows that the bound $g\ge 2r+3$   for the irreducibility of $\mathcal{H}^\mathcal{L} _{g+r-3,g,r}$ is sharp for every $r\ge 5$; cf. Theorem \ref{JPA} (4).
\begin{thm}\label{reducible} For $r\ge 5$,  $\mathcal{H}^\mathcal{L}_{g+r-3,g,r}$ is reducible for every $g$ in the range $r+7\le g\le 2r+2$ having  component(s) dominating the loci $\h{M}^1_{g,k}$ for every $4\le k\le
[\frac{g-r+1}{2}]$.
\end{thm}
\begin{proof} 
We retain all the notations used in the proof of Theorem \ref{g+r-2}.
Since $g\ge r+7$, we have $e:=g-r+1\ge 8$. We first fix an integer $k\ge 4$ such that $e=g-r+1\ge 2k$ and set $\delta:=g-r+1-2k\ge 0$.  On a general $k$-gonal curve $C$ with a unique $g^1_k$, $$|K_C-2g^1_k|=g^{g-2k+1}_{2g-2-2k}$$ is very ample by \cite[Corollary 3.3]{Keem}.  Furthermore, we claim that for a general choice of $\Delta\in C_\delta=C_{g-r+1-2k}$, the complete linear series  
$$|K_C-2g^1_k-\Delta|=g^{g-2k+1-\delta}_{2g-2-2k-\delta}=g^r_{g+r-3}$$ is very ample as long as $k\ge 4$. We take $$\Delta=r_1+\cdots +r_\delta\in C_\delta$$ such that 
no pair $\{ r_i, r_j\}$ lies on the same fiber of the $k$-sheeted covering over $\PP^1$. Upon making  a choice of such $\Delta$, for an arbitrary choice $s+t\in C_2$, we take $ D=\Delta +s+t$,  $n=2$ and $m=\delta +2$ in Lemma \ref{kveryample}.  
By the assumption $g\le 2r+2$, we see that the numerical assumption (\ref{veryamplek}) in Lemma \ref{kveryample} ( $g\ge 2m+n(k-1)$) is satisfied.  By our choice of $\Delta$, we have  $\deg\gcd (D,E)\le 3$ for any $E\in g^1_k$ and hence there is no $E\in g^1_k$ with $E\le D$ as long as  $k\ge 4$.  Therefore it follows that
$$\dim|2g^1_k+\Delta+s+t|=2, $$
for any $s+t\in C_2$ implying $|K_C-2g^1_k-\Delta|$ is indeed very ample. 
Let $$\h{F}_{k,\delta}:=(2\h{G}^1_k+\h{W}_\delta)^\vee\subset\h{G}^r_{g+r-3}$$ 
 be the locus consisting of residual series of nets $\h{E}\in\h{W}^2_{g-r+1}$ of the form $$\h{E}=|2g^1_k+\Delta |,~ \Delta\in C_\delta$$ on a general $k$-gonal curve $C\in\h{M}^1_{g,k}$.
 The following inequality holds by the assumption $g\le 2r+2$;
\begin{equation}\label{kcom}
\begin{split}\lambda(g+r-3,g,r)&=4g-3r-6\le\dim\h{F}_{k,\delta}=\dim(2\h{G}^1_k+\h{W}_\delta)\\&=\dim\h{G}^1_k\underset{\h{M}^1_{g,k}}\times\h{W}_\delta\\&=\dim\h{G}^1_{k}+\delta\\&=3g-3+\rho(k,g,1)+\delta\\
&=3g-3+(2k-g-2)+(g-r+1-2k)\\
&=3g-4-r.
\end{split}
\end{equation}
It is a priori possible that these loci $\h{F}_{k,\delta}$ (corresponding to each $k\ge 4$ and $\delta$ such that $2k+\delta=g-r+1$) might be in the boundary of some other component of bigger dimension. However this would be impossible for which we argue as follows. Let $\h{G}\subset\h{G}^r_{g+r-3}$ be a component properly containing the locus $\h{F}_{k,\delta}$.
 If the base-point-free part of the  residual series $\h{E}$  of a general $\h{D}\in\h{G}$ is birationally very ample,  
by an elementary dimension count using  Proposition \ref{wrdbd} (b)
$$\lambda (g+r-3,g,r)\le \dim \h{W}^\vee=\dim\h{G}\le 3(g-r+1)+g-9-2b$$ implying $b=0$ and $\dim\h{G}=\lambda(g+r-3, g,r)$ where $b$ is the degree of the base locus of $\h{E}$. Hence we have 
$$\dim\h{F}_{k,\delta}=\dim (2\h{G}^1_k+\h{W}_\delta)^\vee =3g-4-r \lneqq \dim\h{G}=4g-3r-6,$$ which is not compatible with above inequality (\ref{kcom}) or the genus assumption $g\le 2r+2$.

\vni 
The remaining possibility is that the residual series $\h{E}$  of a general $\h{D}\in\h{G}$ is compounded and induces an $n$-sheeted covering $C\stackrel{\pi}\rightarrow E$ onto an irrational curve $E$ of genus $h\ge 1$. However this is not possible by Remark \ref{kdelta} (iv).
Therefore we may deduce that the irreducible locus $\h{F}_{k,\delta}=(2\h{G}^1_k+\h{W}_\delta)^\vee$ is indeed dense in a component 
$\h{G}\subset\widetilde{\h{G}}_\h{L}$,
which
gives rise to a component of $\h{H}^\h{L}_{g+r-3, g,r}$ of  dimension, 
$$\dim\h{M}^1_{g,k}+\delta+\dim\PP\textrm{GL}(r+1)=3g-4-r+\dim\PP\textrm{GL}(r+1)$$
as a $\PP\textrm{GL}(r+1)$-bundle over the irreducible locus $\h{F}_{k,\delta}$.
\vni In all, we conclude that there is a component $\h{H}_{k,\delta}\subset \h{H}^\h{L}_{g+r-3}$ sitting over the component $\h{F}_{k,\delta}\subset\widetilde{\h{G}}_\h{L}$ dominating 
$\h{M}^1_{g,k}$ of equal dimension $3g-4-r$ for each $k$'s in the range $4\le k\le [\frac{g-r+1}{2}]$;
$$\h{H}_{k,\delta} \dasharrow \h{F}_{k,\delta}\dasharrow \h{M}^1_{g,k}\subset\h{M}_g.$$
 \vni
\begin{enumerate}
\item[(i)] For $g=r+7$, $\h{H}^\h{L}_{g+r-3,g,r}$ is reducible by Theorem \ref{g=r+7}.
\vni
\item[(ii)] 
For $r+9\le g\le 2r+2$, there are at least $l(r,g):=[\frac{g-r+1}{2}]-3\ge 2$ components $\h{H}_{k,\delta}$'s for each $k$ in the range
$8\le 2k\le g-r+1$.
\vni
\item[(iii)] We assume $g=r+8$ and $r\ge 6$. Since we are working in the range $r+5\le g\le 2r+2$, we have $r\ge 6$ if $g=r+8\le 2r+2$.
\vni
\item[(a)]
If the residual series $\h{E}=g^2_{g-r+1}=g^2_9\in\h{W}^\vee$ of  a general element $\h{D}\in\h{G}\subset\widetilde{\h{G}}_\h{L}$ is compounded and induces a $4$-sheeted covering onto $\PP^1$,  we have already seen that there is a component $\h{H}_{4,1}\subset\h{H}^\h{L}_{g+r-3,g,r}$ sitting over the irreducible family $\h{F}_{4,1}=(2\h{G}^1_4+\h{W}_1)^\vee$ dominating $\h{M}^1_{g,4}$.
\vni 
\item[(b)]
We may add another specific irreducible family $\h{F}_{(2,3),1}\subset\h{G}\subset\widetilde{\h{G}}_\h{L}$ by claiming that for a double cover $C\stackrel{\pi}{\rightarrow}E$ onto a non-hyperelliptic curve $E$ genus $h=3$, the linear series 
$$|K_C-\pi^*(K_E)-q|$$ is very ample for a general $q\in C$ if $g=r+8\ge 21$, i.e. $r\ge 13$. We let $\h{E}=|\pi^*(K_E)+q|=g^2_9$ and assume the existence of  $D\in C_2$ such that $$\dim|\h{E}+D|=\dim|\pi^*(K_E)+q+D|\ge \dim\h{E}+1=3. $$  
Recall that $|K_C-\pi^*(K_E)|$ is very ample as we have seen in the course of the proof of Theorem \ref{g=r+7} and hence  $|\pi^*(K_E)+D|=g^2_{10}$.
Therefore $$|\pi^*(K_E)+q+D|=g^3_{11}$$
 is base-point-free and birationally very ample (but not  very ample)  inducing a birational morphism onto a space curve $\tilde{C}\subset\PP^3$.
Note that \[g=r+8\le p_a(\tilde{C})\le \pi (11,3)=20\]
contrary to the assumption $g=r+8\ge 21$. 
Therefore it follows that $|K_C-\pi^*(K_E)-q|$ is very ample finishing the proof of the claim.
By the usual dimension count, the family $\h{F}_{(2,3),1}\subset\h{G}\subset\widetilde{\h{G}}_\h{L}$ consisting of very ample linear series of the form $|K_C-\pi^*(K_E)-q|$ over a component of  $\h{X}_{2,3}\subset\h{M}_g$ has dimension
$$\dim\h{F}_{(2,3),1}=\dim\h{X}_{2,3}+1=2g-3\ge 4g-3r-6=\lambda (d,g,r);$$ where the inequality holds by the assumption $g=r+8\ge 21$. 
We see that $\h{F}_{(2,3),1}$ is not in the boundary of $\h{F}_{4,1}$ by (lower) semi-continuity of gonality. Hence for $g=r+8\ge 21$ there are at least two components of $\widetilde{\h{G}}_\h{L}\subset\h{G}^r_{g+r-3}$; $\h{F}_{4,1}$ dominating $\h{M}^1_{g,4}$, another one containing the family $\h{F}_{(2,3),1}$.
 We denote\footnote{ The notation $\h{H}_{(2,3),1}$ is used in order to avoid confusion with $\h{H}_{(2,3)}$ which were used in the proof of Theorem \ref{g=r+7}.} by $\h{H}_{(2,3),1}$, the irreducible family in $\HL{2r+5,r+8,r}$ corresponding to the family $\h{F}_{(2,3),1}$ and by $\widetilde{\h{H}}_{(2,3)1}$, the component containing $\h{H}_{(2,3),1}$.
\vni
\item[(c)]
For $6\le r\le 12$ (in fact for $6\le r\le 20$) one may come up with another component $\h{H}_{\Sigma_{9,g}}\subset\h{H}^\h{L}_{g+r-3,g,r}$ corresponding to the Severi 
variety $\Sigma_{9,g}$. Indeed we may copy the proof  Theorem \ref{g=r+7} (ii) with only trivial modifications as follows. In the current situation, we have 
$$e=\deg\h{E}=\deg|K_C-\h{D}|=9$$ hence the equality (\ref{delta}) in the proof (ii) of Theorem \ref{g=r+7} becomes
\begin{equation}\label{e=9}
 \delta=\frac{(e-1)(e-2)}{2}-g=28-(r+8)=20-r
 \end{equation}  
 Then for a general  $\Delta =\{p_1, \cdots , p_{\delta}\}\in{\text Sym}^\delta(\PP^2),$
 the inequality 
$$~~\quad \delta=20-r \le {\rm min} \{e(e+3)/6,(e-1)(e-2)/2\} \text {~and~}  (e,\delta)\neq (6,9)$$
-- which is virtually the inequality (\ref{AC}) -- continues to hold for $e=9$ and $6\le r\le 20$.
 Then one invokes a result by Arbarello and Cornalba \cite[Theorem 3.2]{AC1} to see eventually that the non-singular model $C\subset S$ -- where $S:=S_{20-r}$ is the  blowing up of $\PP^2$ at $\Delta$ -- of a nodal plane curve $E$ of degree $e=g-r+1=9$ in the linear system 
 $$|C|=|\pi^*(eL)-\sum_{i=1}^\delta 2E_i|$$ is embedded into $\PP^r$ as a linearly normal curve of degree $d=g+r-3$ by the very ample linear 
system  
$$|H|=|\pi^*((e-4)L)-\sum_{i=1}^\delta E_i)|.$$
\noindent
To sum up, for $g=r+8$ and $r\ge 6$ we  located the following components, which {\bf may not be all} the components of $\h{H}^\h{L}_{g+r-3,g,r}$: cf. Remark \ref{triple}.
\vni
$\h{H}_{4,1}$ and $\h{H}_{\Sigma_{9,r+8}}$ \quad if $6\le r\le 12$.
\vni
$\h{H}_{4,1}$, $\widetilde{\h{H}}_{(2,3),1}$ and $\h{H}_{\Sigma_{9,r+8}}$ \quad if $13\le r\le 20$. 
\vni
$\h{H}_{4,1}$ and $\widetilde{\h{H}}_{(2,3),1}$ \quad if $r\ge 21$.\end{enumerate}
\end{proof}

\vni
Theorem \ref{reducible} together with Remark \ref{extremal} (ii) yield the following immediate corollary which identifies all the triples $(g+r-3.g.r)$ with $r\ge 5$ such that $\HL{g+r-3,g,r}$ is reducible.
\begin{cor} Let  $\Upsilon\subset \mathbb{N}^{\oplus 3}$ be same as in Theorem \ref{first} and  we set 
$$\widetilde{\Upsilon}=\{(g+r-3,g,r)| r+7\le g\le 2r+2, r\ge 5\}\cup \{ (12,10,5)\}\subset\Upsilon.$$
For $r\ge 5$, $$\HL{g+r-3,g,r} \textrm{ is reducible if and only if }  (g+r-3,g,r)\in\widetilde{\Upsilon}.$$ 
\end{cor}

\begin{rmk}\label{triple} 
\begin{enumerate}
\item[(i)] When we were dealing with the case $g=r+8$ in the proof  of Theorem \ref{reducible} (iii), we ignored the possibility that $\h{E}=g^2_9$ may induce a 
triple covering $C\stackrel{\pi}{\rightarrow} E$ onto an elliptic curve $E$ and $\h{E}=\pi^*{(g^2_3)}$. Indeed, whenever $g=r+8\ge 16$, one may verify easily that the residual series of $\h{E}=\pi^*{(g^2_3)}$ is  very ample by using Castelnuovo-Severi inequality, etc.
However for  $6\le r\le 10$,  the family of linear series of the form $$
\h{S}:=\{ \h{E}=\pi^*{(g^2_3)}| ~C\stackrel{\pi}{\rightarrow} E, C\in\h{X}_{3,1})\}\subset \h{W}^2_9$$
has dimension
\begin{eqnarray*}
 \dim\h{X}_{3,1}&+&\dim W^2_3(E)=\dim\h{X}_{3,1}+\dim J(E)\\&=&\dim\h{X}_{3,1}+1
 \lneq \lambda(g+r-3,g,r)=4g-3r-6,
 \end{eqnarray*}
  where $E$ is an elliptic curve, hence the family $\h{S}^\vee$ is not a component of $\widetilde{\h{G}}_\h{L}$. When $r\ge 11$ at least one extra component $\h{H}_{(3,1)}$ corresponding to  triple coverings of elliptic curves arises.
 \item[(ii)] We remark that when $g=r+8$ and $r\ge 13$, $\widetilde{\h{H}}_{(2,3),1}\neq\h{H}_{(3,1)}$. Otherwise, ${\h{H}}_{(2,3),1}\subsetneq\h{H}_{(3,1)}$ by simple dimension counting. By Castelnuovo-Severi inequality,  a smooth double cover of a curve of genus three cannot be a  of triple cover of elliptic curve. Hence a smooth double cover of a curve of genus three is a limit of triple covers of elliptic curves which is impossible by the theory of admissible covers for a compactification of Hurwitz spaces of branched coverings.\footnote{The author has been informed by the referee that any smooth limit of $k_1$-fold covers of genus $h_1$ curves must be a $k_2$-fold
cover of a genus $h_2$ curve for some $k_2\leq k_1$ and $h_2 \leq h_1$; from the theory of admissible covers,  when branched covers degenerate one can keep the branch points separated and allow the domain/target curves to become nodal.  In this process the (arithmetic) genus remains the same, but the domain/target curves may become reducible whereas each irreducible component may have smaller genus and the degree of the cover restricted to a component may possibly become smaller;  cf. \cite[3G]{H3}  and \cite{HM}.}
  \item[(iii)] It is also worthwhile to mention that the component $\widetilde{\h{H}}_{(2,3),1}$ coincide with the irreducible locus ${\h{H}}_{(2,3),1}$. Note that given a very ample (and general) element $\h{D}\in\h{G}\subset\h{G}_\h{L}\subset\h{G}^r_{g+r-3}=\h{G}^r_{2r+5}$, the following four cases are all the possibilities for the residual series $\h{E}=\h{D}^\vee=g^2_9$: 
 
\begin{enumerate}
 \item[(a)]
$\h{E}$ is base-point-free and  birationally very ample; in this case,  $\h{G}^\vee =\h{G}'\subset\h{G}^2_9$ which the Severi variety $\Sigma_{9,r+8}$ as well as the component $\h{H}_{\Sigma_{9,r+8}}$ of expected dimension sits over.

 \item[(b)]
$\h{E}$ is base-point-free and compounded inducing a triple cover of an elliptic curve; in this case we have the component $\h{H}_{(3,1)}$.

 \item[(c)]
$\h{E}$ is not base-point-free and compounded inducing a double cover of a smooth plane quartic; the 
irreducible locus $\h{H}_{(2,3),1}$ corresponding to the family  $\h{F}_{(2,3),1}\subset\h{G}$ (introduced in the proof of Theorem \ref{reducible})  arises.

 \item[(d)]
$\h{E}$ is not base-point-free and compounded inducing a $4$-sheeted cover of a rational curve;   the corresponding component is $\h{H}_{4,1}$ dominating $\h{M}^1_{g,4}$. 
\end{enumerate}

\vni Since $\widetilde{\h{H}}_{(2,3),1}\neq\h{H}_{(3,1)}$, the locus $\h{F}_{(2,3),1}$ is indeed dense in $\h{G}$ and therefore $\widetilde{\h{H}}_{(2,3),1}={\h{H}}_{(2,3),1}$.

\vni
\item[(iv)] If $g=2r+2$, $r$ is odd and $r\ge 7$ in Theorem \ref{reducible}, by taking $k=\frac{r+3}{2}$, $\delta=g-r+1-2k=0$ we have
$$\dim \pi (\h{H}_{k,0})=\dim\h{M}^1_{g,k}=2g+2k-5=\lambda(g+r-3,g,r), $$  where $\h{H}_{k,0}\stackrel{\pi}{\dasharrow}\h{M}_g$ is the functorial map. For every other integer $k$ with $4\le k\lneqq\frac{r+3}{2}$, the component  $\h{H}_{k,\delta}$ has less than the expected number of moduli even though all the components $\h{H}_{k,\delta}$'s  have the same expected dimension. 
\item[(v)] If $g=2r+2$, $r$ is even and $r\ge 8$, every component $\h{H}_{k,\delta}$ has less than the expected number of moduli.
\end{enumerate}
\end{rmk}

\vni For the convenience of readers we add the following table which is a summary of rather complicated situation for the case $g=r+8$ in Theorem \ref{reducible} by taking our discussion in Remark \ref{triple} into account.
\vfill
\pagebreak
\begin{table}[ht]
\caption{Components of $\h{H}^\h{L}_{g+r-3,g,r}$ for $g=r+8$} 
\centering 
\begin{tabular}{c c c c c c} 
\hline\hline 
{\footnotesize Dimension}& {\footnotesize\# of }  & {\footnotesize\# of compts} & {\footnotesize \# of compts} & {\footnotesize \# of compts} & {\footnotesize \# of compts}\\ [0.5ex] 
{\footnotesize of $\PP^r$}  & {\footnotesize compts} & {\footnotesize of expected} & {\footnotesize with expected} & {\footnotesize $\lneqq$ expected} &{\footnotesize $\gneqq$ expected}\\ [0.5ex] 
{}  & {} & {\footnotesize dimension} & {\footnotesize \# of moduli} & {\footnotesize \# of moduli} &  {\footnotesize \# of moduli}\\ [0.5ex] 
\hline 
\hline 
{\footnotesize $r=6$} & {\footnotesize 2} & {\footnotesize 2} & {\footnotesize 1} & {\footnotesize 1} & {\footnotesize 0}\\
{\footnotesize $r=7$} & {\footnotesize 2} & {\footnotesize 1} & {\footnotesize 2} & {\footnotesize 0}  &{\footnotesize 0}\\ 
{\footnotesize $8\le r \le 10$} & {\footnotesize 2}& {\footnotesize 1} & {\footnotesize 1} & {\footnotesize 0} & {\footnotesize 1} \\
{\footnotesize $r=11$} & {\footnotesize $\ge 3$}& {\footnotesize $\ge 2$} & {\footnotesize 1} & {\footnotesize $\ge 1$} & {\footnotesize 1} \\
{\footnotesize $r=12$} & {\footnotesize $\ge 3$}& {\footnotesize 1} & {\footnotesize $\ge 2$} & {\footnotesize 0} & {\footnotesize 1} \\
{\footnotesize $r=13$} & {\footnotesize $\ge 4$}& {\footnotesize $\ge 2$} & {\footnotesize $1$}& {\footnotesize $\ge 1$} & {\footnotesize $\ge 2$} \\ 
{\footnotesize $r=14$} &{\footnotesize $\ge 4$} & {\footnotesize $1$} & {\footnotesize $\ge 2$} &{\footnotesize $0$} & {\footnotesize $\ge 2$}\\
{\footnotesize $15\le r\le 20$} & {\footnotesize $\ge 4$}  & {\footnotesize $1$} & {\footnotesize $1$} & {\footnotesize $ 0$}  & {\footnotesize $\ge 3$}  \\
{\footnotesize $r\ge 21$}& {\footnotesize $\ge 3$}  & {\footnotesize $0$} & {\footnotesize $0$} & {\footnotesize $0$} &{\footnotesize $\ge 3$} \\ [1ex] 
\hline 
\end{tabular}
\label{table:nonlin} 
\end{table}

\bibliographystyle{spmpsci} 

\end{document}